\documentclass[10pt]{article}
\usepackage{amsmath,amssymb,amsfonts}
\usepackage{amsthm}
\usepackage{dsfont}
\usepackage{amssymb,latexsym}
\usepackage{enumerate}

\newcommand{\nc}{\newcommand}
\nc{\CC}{\mathds{C}}
\nc{\RR}{\mathds{R}}
\nc{\ZZ}{\mathds{Z}}
\nc{\NN}{\mathds{N}}
\nc{\g}{\mathfrak{g}}
\nc{\U}{\mathcal{U}}
\nc{\p}{\mathfrak{p}}
\nc{\D}{\mathcal{D}}
\nc{\Da}{\mathcal{D}^a}
\nc{\Dp}{\mathcal{D}_p}
\nc{\Dap}{\mathcal{D}_p^a}
\nc{\gl}{\widehat{\mathfrak{gl}}_{\infty}^{[m]}}
\nc{\glm}{\mathfrak{gl}_{\infty}^{[m]}}
\nc{\li}{\mathcal{L}}
\nc{\Rm}{\mathrm{R_m}}
\nc{\ho}{\mathcal{O}}
\nc{\fis}{\varphi_s^{[m]}}
\nc{\fiss}{\widehat{\varphi}_s^{[m],\,\pm}}

\numberwithin{equation}{section}

\theoremstyle{plain}
\newtheorem{teor}{Theorem}[section]
\newtheorem{prop}[teor]{Proposition}
\newtheorem{corol}[teor]{Corollary}
\newtheorem{lema}[teor]{Lemma}
\newtheorem{defi}[teor]{Definition}
\newtheorem{obse}[teor]{Remark}

\title{Quasifinite Representations of Classical Lie subalgebras of $W_{\infty,\,p}$}
\author{Jos\'{e} I. Garc\'{i}a and Jos\'{e} I. Liberati}
\date{27 June 2012}

\begin{document}

\maketitle

\begin{abstract}
We show that there are exactly two anti-involution $\sigma_{\pm}$ of the algebra of differential operators on the circle that are a multiple of $p(t\partial_t)$ preserving the principal gradation ($p\in\CC[x]$ non-constant). We classify the irreducible quasifinite highest weight representations of the central extension $\widehat{\D}_p^{\pm}$ of the Lie subalgebra fixed by $-\sigma_{\pm}$. The most important cases are the subalgebras $\widehat{\D}_x^{\pm}$ of $W_{\infty}$, that are obtained when $p(x)=x$. In these cases we realize the irreducible quasifinite highest weight modules in terms of highest weight representation of the central extension of the Lie algebra of infinite matrices with finitely many non-zero diagonals over the algebra $\CC[u]/(u^{m+1})$ and its classical Lie subalgebras of $C$ and $D$ types.
\end{abstract}

\section{Introduction}
The universal central extension $\widehat{\mathcal{D}}$ of the Lie algebra of differential operators on the circle (described first in \cite{kp}) is usually denoted by physicists as $W_{1+\infty}$, and it is one of the $\mathcal{W}$- infinity algebras that naturally arise in various physical theories, such as conformal field theory, the theory of quantum Hall effect, etc.

 The difficulty in understanding the representation theory of a Lie algebra of this kind is that although it admits a $\mathbb{Z}$-gradation (and thus the associated triangular decomposition), each of the graded subspaces is still infinite dimensional, and therefore the study  of highest weight modules with the finiteness requirement on the dimensions of their graded subspaces (which we will refer to as quasifiniteness condition)  become
 a non-trivial problem.

 The study of quasifinite highest weight modules of $\widehat{\mathcal{D}}$ was initiated by Kac and Radul \cite{kr1}. They were able to give a characterization of its irreducible quasifinite highest weight representations and these modules were constructed in terms of modules of the Lie algebra $\widehat{\mathfrak{gl}}_{\infty}^{[m]}$ which is the central extension of the Lie algebra ${\mathfrak{gl}}_{\infty}^{[m]}$ of infinite matrices with finitely many non-zero diagonals taking values in the truncated polynomial algebra $R_m=\mathbb{C}[u]/(u^{m+1})$. On the basis of this analysis, further studies were made within the framework of vertex algebra theory for the $\widehat{\mathcal{D}}$ algebra \cite{fkrw,kr2}, and for its matrix version \cite{bkly}. In \cite{kl} a general approach to the theory  of quasifinite highest weight modules over $\mathbb{Z}$-graded Lie algebras was developed, which makes  the basic ideas of \cite{kr1} much clearer, and these general results will be applied here. In \cite{afmo} and \cite{kl}, they develop the theory of quasifinite highest weight representations of the subalgebras $\widehat{\mathcal{D}}_p$ of  $\widehat{\mathcal{D}}$, where $\widehat{\mathcal{D}}_p$ ($p\in\mathbb{C}[x]$) is the central extension of the Lie algebra $\mathcal{D}p(t\partial_t)$ of differential operators on the circle that are a multiple of $p(t\partial_t)$. The most important of these subalgebras is $W_\infty=\widehat{\mathcal{D}}_x$ that is obtained by taking $p(x)=x$. Classical Lie subalgebras of $\widehat{\mathcal{D}}$ appear by the study of anti-involutions on $\widehat{\mathcal{D}}$. The orthogonal subalgebras of $\widehat{\mathcal{D}}$ were studied in \cite{kwy}. The symplectic subalgebra of $\widehat{\mathcal{D}}$ was considered in \cite{b} in relation to number theory, and the  representation theory was developed in \cite{bl}.

 The idea of the present work is to extend some results from \cite{kwy} to the family of subalgebras $\widehat{\mathcal{D}}_p$. More precisely, in section 2 we show that  there are exactly two, up to conjugation, anti-involutions $\sigma_\pm$ of $\mathcal{D}_p^a$ preserving the principal gradation. In section 3, we classify the irreducible quasifinite highest weight representations of the central extension $\widehat{\mathcal{D}}_p^\pm$ of the Lie subalgebra of $\widehat{\mathcal{D}}_p$ fixed by $-\sigma_\pm$. In particular, if $p=1$, from our results we recover several theorems  obtained in \cite{kwy}. The other most important cases are the subalgebras $\widehat{\mathcal{D}}_x^\pm$ of $W_{\infty}$, that are obtained by taking $p(x)=x$. For these cases, in section 4 we study the interplay between $\widehat{\mathcal{D}}_x^\pm$ and some subalgebras of $\widehat{\mathfrak{gl}}_{\infty}^{[m]}$, and in section 5, we realize the irreducible quasifinite highest weight representations in terms of highest weight modules of the Lie algebra $\widehat{\mathfrak{gl}}_{\infty}^{[m]}$ and its classical Lie subalgebras of  $C$ and $D$ types.  Observe that the symplectic subalgebra of $\widehat{\mathcal{D}}$ considered in \cite{b} and  \cite{bl} is a particular case of our general results, it corresponds to $\widehat{\mathcal{D}}_x^+$.

\section{Anti-Involution of $\Dap$ Preserving its Principal Gradation}

\qquad Let $\Da$ be the associative algebra of regular differential operators on the circle, i.e. the operators on $\CC[t,t^{-1}]$ of the form
\[E=e_k(t)\partial_t^k+e_{k-1}(t)\partial_t^{k-1}+\cdots+e_0(t), \ \text{where} \ e_i(t)\in\CC[t,t^{-1}],\]
the elements
\[J_k^l=-t^{k+l}\partial_t^l \quad (l\in\ZZ_+, \ k\in\ZZ),\]
form its basis, where $\partial_t$ denotes $\frac{d}{dt}$. Another basis of $\Da$ is
\[L_k^l=-t^kD^l \quad (l\in\ZZ_+, \ k\in\ZZ),\]
where $D=t\partial_t$. Let $\D$ denote the Lie algebra obtained from $\Da$ by taking the usual bracket, i.e.
\[ [t^rf(D),t^sg(D)]=t^{r+s}(f(D+s)g(D)-f(D)g(D+r)), \]
where  $f, \ g\in\CC[x]$ and $s, \ t\in\ZZ$. Given $p\in\CC[x]$, consider the following family of subalgebras of $\Da$,
\[\Dap:=\Da p(D)\]
and denote by $\Dp$ the associated Lie algebra (cf. \cite{kl}).

Letting $\mathrm{wt} \ t^kf(D)=k$ defines the principal $\ZZ$-gradation of $\Da$ and $\Dap$:
\[\Dap=\bigoplus\limits_{j\in\ZZ}(\Dap)_j, \ \text{where} \ (\Dap)_j=\{t^jf(D)p(D):f\in\CC[x]\}.\]

An \textit{anti-involution} $\sigma$ of $\Dap$ is an involutive anti-automorphism of $\Dap$, i.e. $\sigma:\Dap\rightarrow\Dap$ with $\sigma^2=Id, \ \sigma(bX+Y)=b\sigma(X)+\sigma(Y) \ and \ \sigma(XY)=\sigma(Y)\sigma(X), \ \text{where} \ X,Y\in\Dap, \ b\in\CC.$

\medskip
The main result of this section is the following theorem with the classification of all anti-involutions of $\Dap$ that preserve the principal $\ZZ$-gradation.

\begin{teor} \label{clasificacion antinvolution}
Let $p\in \CC[x]$ be a non-zero polynomial. There exist an anti-involution in $\Dap$ that preserve the principal $\ZZ$-gradation if and only if exist $c \in \CC \ \text{such that} \ p(x)=\varepsilon p(-x+c)$, where $\varepsilon=(-1)^{deg(p)}$.

If $deg(p)\geqslant1$, then $c$ is unique and there exist only two anti-involutions given by
\begin{equation} \sigma_{\pm}(t^kf(D)p(D))=\varepsilon(\pm t)^kf(-D-k+c)p(D). \label{formulasigma} \end{equation}

If $deg(p)=0$, then $c$  is a free parameter, and there are only two families of anti-involutions given by
\eqref{formulasigma}.

\end{teor}

\medskip
\begin{obse} When $deg(p)=0$, we recover the classification obtained in Proposition 2.1 \cite{kwy}. \end{obse}

In the last part of this section we present the proof of Theorem \ref{clasificacion antinvolution} throughout several lemmas.

\smallskip
Let $\sigma:\Dap\rightarrow\Dap$ be an anti-involution that preserve the principal gradation, then $\sigma$ induce a map $\sigma_0:\Da\rightarrow\Da$ as follows
\begin{equation} \sigma(t^kf(D)p(D))=\sigma_0(t^kf(D))p(D). \label{sigma=sigma0p} \end{equation}
It's clear that $\sigma_0$ preserves the principal gradation and furthermore, the characterization of $\sigma$ is equivalent to the characterization of $\sigma_0$.

\begin{lema}\label{propiedades sigma0} Let $f, \ g \in\CC[D]$ and $k, \ m \in\ZZ$. Then
 \begin{enumerate}[(a)]
  \item $\sigma_0$ is $\CC$-linear; \label{propiedades sigma0 a}
  \item $\sigma_0^2=Id$;\label{propiedades sigma0 b}
  \item $\sigma_0(t^{k+m}f(D+m)p(D+m)g(D))=\sigma_0(t^mg(D))p(D)\sigma_0(t^kf(D))$; \label{propiedades sigma0 c}
  \item $\sigma_0(f(D)g(D)p(D))=\sigma_0(f(D))\sigma_0(g(D))p(D)$ \label{propiedades sigma0 d}.
 \end{enumerate}
\end{lema}

\begin{proof}
Using that $\sigma$ is an anti-involution, \eqref{propiedades sigma0 a} and \eqref{propiedades sigma0 b} follows immediately. For $f, \ g \ \in\CC[D], \ k, \ m\in\ZZ$ we have
\begin{equation} \sigma(t^kf(D)p(D)t^mg(D)p(D))=\sigma_0(t^mg(D))p(D)\sigma_0(t^kf(D))p(D) \label{(17)} \end{equation}
and
\begin{equation} \sigma(t^kf(D)p(D)t^mg(D)p(D))=\sigma_0(t^{k+m}f(D+m)p(D+m)g(D))p(D) \label{(18)} \end{equation}
obtaining \eqref{propiedades sigma0 c}. Observe that \eqref{propiedades sigma0 d} follows from \eqref{propiedades sigma0 c} since $(\Da)_0$ is an abelian sub-algebra of $\Da$ and $\sigma_0$ preserve the gradation, finishing the proof.
\end{proof}

We shall need the following notation: $\sigma_0(t^k)=t^k \varepsilon_k$, with $\varepsilon_k$ in $\CC[D]$.

\begin{lema}\label{determino sigma0}
 \begin{enumerate}[(a)]
  \item For all $k\in\ZZ$, we have $\varepsilon_k=\pm1$. \label{determino sigma0 a}
  \item There exist $c\in\CC$ such that for all $k\in\ZZ$ and $f\in\CC[D]$, we have
  \[ \sigma_0(t^kf(D))=\varepsilon_k t^k f(-D-k+c). \] \label{determino sigma0 b}
 \end{enumerate}
\end{lema}

\begin{proof}
Using Lemma \ref{propiedades sigma0}.\eqref{propiedades sigma0 d} with $f=g=1$ we have
\begin{equation} \sigma_0(p(D))=\varepsilon_0^2 p(D) \label{s(p)=e^2p} \end{equation}
then by Lemma \ref{propiedades sigma0}.\eqref{propiedades sigma0 b} and \eqref{s(p)=e^2p}, $p(D)=\sigma_0(\varepsilon_0^2p(D))$ and using again Lemma \ref{propiedades sigma0}.\eqref{propiedades sigma0 d} with $f=1$ and $\ g=\varepsilon_0^2$, we obtain
\begin{equation} p(D)=\varepsilon_0 \sigma_0(\varepsilon_0^2) p(D).\label{(51)} \end{equation}
Then from \eqref{(51)}, $1=\varepsilon_0 \sigma_0(\varepsilon_0^2)$ therefore $\varepsilon_0$ is a constant. Moreover by Lemma \ref{propiedades sigma0}.\eqref{propiedades sigma0 a} and \eqref{propiedades sigma0 b}, we have $1=\sigma_0^2(1)=\varepsilon_0^2$, obtaining
 \begin{equation} \varepsilon_0=\pm1. \label{(step1)} \end{equation}

Now we shall prove that
\begin{equation}\sigma_0(t^l D^i)=t^l\varepsilon_l\,(\varepsilon_0\sigma_0(D)-l)^i \ \text{for all} \ l\in\ZZ, \ \text{and} \ i\in\ZZ_+, \label{(step2)} \end{equation}
by using induction in $i$. The case $i=0$ follows by notation. Now, using Lemma \ref{propiedades sigma0}.\eqref{propiedades sigma0 c} with $t^kf(D)=(D-l), \ t^mg(D)=t^lD^i$, we have
\begin{eqnarray}
 \sigma_0(t^lD^{i+1}p(D+l))&=&\sigma_0(t^lD^i)p(D)\sigma_0(D-l) \notag \\
                           &=&t^l\varepsilon_l\,(\varepsilon_0\sigma_0(D)-l)^ip(D)(\sigma_0(D)-\varepsilon_0l) \label{(1)}
\end{eqnarray}
on the other hand, using again Lemma \ref{propiedades sigma0}.\eqref{propiedades sigma0 c} with $t^kf(D)=1, \ t^mg(D)=t^lD^{i+1}$, we obtain
\begin{equation} \sigma_0(t^lD^{i+1}p(D+l))=\sigma_0(t^lD^{i+1})p(D)\varepsilon_0. \label{(2)} \end{equation}
Comparing \eqref{(1)} with \eqref{(2)}, and using \eqref{(step1)}, we obtain \eqref{(step2)}.

Note that, given $f\in\CC[x]$ and using the linearity of $\sigma_0$ together with \eqref{(step2)}, we have
\begin{equation} \sigma_0(t^kf(D))=t^k\varepsilon_kf(\varepsilon_0\sigma_0(D)-k). \label{(3)} \end{equation}

Since $\sigma_0$ preserves the $\ZZ$-gradation, we can assume that $\sigma_0(D)=g(D)$ for some $g\in\CC[x]$. Then by Lemma \ref{propiedades sigma0}.\eqref{propiedades sigma0 b} and \eqref{(3)}, we have $D=\sigma_0^2(D)=\varepsilon_0g(\varepsilon_0g(D))$. Using \eqref{(step1)}, it follow that $deg(g)^2=1$ and therefore $deg(g)=1$. Then $\sigma_0(D)=aD+b$ for some $a, \ b \in\CC$ with $a\neq0$. Finally,
\[ D=\sigma_0^2(D)=\sigma_0(aD+b)=a^2D+(a+\varepsilon_0)b, \]
obtaining
\begin{equation} \sigma_0(D)=aD+b \ \text{with} \ a=\pm1, \ (a+\varepsilon_0)b=0. \label{(step3)} \end{equation}

Using Lemma \ref{propiedades sigma0}.\eqref{propiedades sigma0 b} and \eqref{(3)}, for all $k\in\ZZ$,
\begin{equation} t^k=\sigma_0^2(t^k)=\sigma_0(t^k\varepsilon_k)=t^k\varepsilon_k(D)\varepsilon_k(\varepsilon_0\sigma_0(D)-k), \label{(20)} \end{equation}
then $deg(\varepsilon_k)=0$ and furthermore $\varepsilon_k^2=1$, finishing the proof of \eqref{determino sigma0 a}.

Observe that, using \eqref{(3)}, \eqref{(step3)} and Lemma \ref{determino sigma0}.\eqref{determino sigma0 a} we have
\[ \sigma_0(t^kf(D))=\varepsilon_k t^k f(\varepsilon_0aD-k+c) \quad \text{where} \quad c=\varepsilon_0b. \]
Hence, in order to finish the proof of \eqref{determino sigma0 b}, it remains to see that $\varepsilon_0a=-1$. But
\begin{eqnarray*}
  t^kD&=&\sigma_0^2(t^kD) \\
      &=&\sigma_0(\varepsilon_k t^k (\varepsilon_0aD-k+\varepsilon_0b)) \\
      &=&t^kD+(\varepsilon_0a+1)k
\end{eqnarray*}
for all $k$ in $\ZZ$, therefore $\varepsilon_0a+1=0$.
\end{proof}

\begin{proof}[Proof Theorem \ref{clasificacion antinvolution}]
Let $\sigma:\Dap\rightarrow\Dap$ be an anti-involution that preserve the principal $\ZZ$-gradation. From \eqref{sigma=sigma0p} and Lemma \ref{determino sigma0}.\eqref{determino sigma0 b}
\begin{equation} \sigma(t^kf(D)p(D))=\varepsilon_kt^kf(-D-k+c)p(D) \label{(21)} \end{equation}
for some $c\in\CC$.

Moreover from \eqref{s(p)=e^2p} and Lemma \ref{determino sigma0}.\eqref{determino sigma0 a} we obtain $\sigma_0(p(D))=p(D)$. Then by Lemma \ref{determino sigma0}.\eqref{determino sigma0 b}, we have that $p$ must satisfy
\begin{equation} p(D)=\varepsilon_0p(-D+c) \ \text{for some} \ c\in\CC. \label{(7)} \end{equation}

If $n=deg(p)>0$, by considering the coefficients of $D^{n}$ and $D^{n-1}$ in both sides of \eqref{(7)}, we have $\varepsilon_0=(-1)^n$ and $c=\dfrac{2c_{n-1}}{nc_n}$ respectively, where $p(x)=\sum_{i=0}^nc_ix^i$, so $c$ is totally determined by the coefficients of $p$.

If $deg(p)=0$, using \eqref{(7)} we get $\varepsilon_0=1$ and $c$ is a free parameter.

On the other hand,
\begin{equation} \sigma(t^kp(D)t^mp(D))=\sigma(t^mp(D))\sigma(t^kp(D)) \label{(4)} \end{equation}
with
\begin{eqnarray}
 \sigma(t^kp(D)t^mp(D))&=&\sigma(t^{k+m}p(D+m)p(D)) \notag \\
                       &=&\varepsilon_{k+m}t^{k+m}p(-D-k+c)p(D) \notag \\
                       &=&\varepsilon_0\varepsilon_{k+m}t^{k+m}p(D+k)p(D), \label{(5)}
\end{eqnarray}
and
\begin{eqnarray}
 \sigma(t^mp(D))\sigma(t^kp(D))&=&\varepsilon_m\varepsilon_kt^mp(D)t^kp(D) \notag \\
                               &=&\varepsilon_m\varepsilon_kt^{k+m}p(D+k)p(D). \label{(6)}
\end{eqnarray}
From \eqref{(5)} and \eqref{(6)} we have that \eqref{(4)} holds if and only if
\begin{equation} \varepsilon_{k+m}=\varepsilon_0\varepsilon_k\varepsilon_m, \label{(52)} \end{equation}
and this is true for all $k, \ m\in\ZZ$. Then, if we take $k=1$ and $m=-1$ from \eqref{(52)} and Lemma \ref{determino sigma0}.\eqref{determino sigma0 a}, we have that $\varepsilon_1=\varepsilon_{-1}$ and by induction
\begin{equation} \varepsilon_k=\varepsilon_0(\varepsilon_0\varepsilon_1)^k \ \text{for all} \ k\in\ZZ. \label{(19)} \end{equation}
Since $\varepsilon_0$ is totally determined, we could have only two anti-involutions depending on the choice of $\varepsilon_1$. Using \eqref{(21)} and \eqref{(19)}, we have the following cases:

\medskip
$\cdot$ if $\varepsilon_1=\varepsilon_0$, then $\sigma(t^kf(D)p(D))=\varepsilon_0t^kf(-D-k+c)p(D)$;

\smallskip
$\cdot$ if $\varepsilon_1=-\varepsilon_0$, then $\sigma(t^kf(D)p(D))=\varepsilon_0(-t)^kf(-D-k+c)p(D)$.

\medskip
Reciprocally, it is straightforward to check that if $p$ satisfies \eqref{(7)} then the two previous cases are anti-involutions, finishing the proof.
\end{proof}

\section{Quasifinite Highest-Weight Modules over $\widehat{\D}_p^{\pm}$}

\qquad Let $p\in\CC[x]$ with $n=deg(p)$ that satisfies Theorem \ref{clasificacion antinvolution}, i.e. $p(x)=(-1)^np(-x+c)$ for some $c\in\CC$. We denote by $\D_p^{\pm}$, the Lie subalgebra of $\D_p$ consisting of its minus $\sigma_{\pm}$-fixed points i.e.,
\[ \Dp^{\pm}=\{d\in\Dp:\sigma_{\pm}(d)=-d\}. \]
It inherits a $\ZZ$-gradation from $\Dp$ since $\sigma_{\pm}$ preserves the principal $\ZZ$-gradation of $\Dp$, then $\D_p^{\pm}=\bigoplus_{k\in\ZZ} (\D_p^{\pm})_k$ where
\[ (\D_p^{\pm})_k=\{t^kf(D)p(D):f\in\CC[x] \ \text{and} \ \sigma_{\pm}(t^kf(D)p(D))=-t^kf(D)p(D)\}. \]

Let us denote by $\CC[x]^{(0)}$ (resp. $\CC[x]^{(1)}$) the set of all even (resp. odd) polynomials in $\CC[x]$. Also, we let $\overline{k}=0$ if $k$ is an odd integer and $\overline{k}=1$ if $k$ even. The following lemma gives a complete description of $(\D_p^{\pm})_k$.

\newpage

\begin{lema} \label{determino Dp+-}
 \begin{enumerate}[(a)]
  \item $(\D_p^+)_k=\left\{t^kf\left(D-\dfrac{c-k}{2}\right)p(D):f\in\CC[x]^{(\overline{n})}\right\}$ \label{determino Dp+};
  \item $(\D_p^-)_k=\left\{t^kf\left(D-\dfrac{c-k}{2}\right)p(D):f\in\CC[x]^{(\overline{n+k})}\right\}$ \label{determino Dp-}.
 \end{enumerate}
\end{lema}

\begin{proof}
 Let $t^kf(D)p(D)\in(\D_p^-)_k$ then, by Theorem \ref{clasificacion antinvolution}
 \[ (-1)^{n+k}t^kf(-D-k+c)p(D)=\sigma_{-}(t^kf(D)p(D))=-t^kf(D)p(D) \]
 if and only if $(-1)^{n+k+1}f(x)=f(-x-k+c)$. We define $g(w)=f(w+\dfrac{c-k}{2})$ and for $x=w-\dfrac{c-k}{2}$, we have $g(-x)=f(-w-k+c)=(-1)^{n+k+1}f(w)=(-1)^{n+k+1}f(x+\dfrac{c-k}{2})=(-1)^{n+k+1}g(x)$, therefore $g(w)\in\CC[w]^{(\overline{n+k})}$ and

 \noindent $g(x-\dfrac{c-k}{2})=f(x)$ finishing \eqref{determino Dp-}. The proof of \eqref{determino Dp+} is similar.
\end{proof}

\begin{obse} \cite{kwy} \end{obse}

We have the following 2-cocycle on $\D$, where $f(x), \ g(x)\in\CC[x]$
\[ \Psi(t^rf(D),t^sg(D))=\begin{cases}
\sum\limits_{-r\leqslant m\leqslant-1}f(m)g(m+r) & \text{if $r=-s>0$}, \\
0 & \text{if $r+s\neq0$ or $r=s=0$}.
\end{cases} \]
Denote by $\widehat{\D}$ the central extension of $\D$ by a one-dimensional center $\CC C$, corresponding to the 2-cocycle $\Psi$, i.e. $\widehat{\D}=\D+\CC C$ with the following commutation relation
\[ [t^rf(D),t^sg(D)]=t^{r+s}(f(D+s)g(D)-f(D)g(D+r))+\Psi(t^rf(D),t^sg(D))C. \]
Denote by $\widehat{\D}_p^{\pm}$ the central extension of $\D_p^{\pm}$ by $\CC C$ corresponding to the restriction of the 2-cocycle $\Psi$.

Letting $wt\,t^kf(D)p(D)=k, \ wt\,C=0$ defines the \textit{principal gradation} of $\widehat{\D}_p^{\pm}$
\begin{equation}
\widehat{\D}_p^{\pm}=\bigoplus_{k\in\ZZ} (\widehat{\D}_p^{\pm})_k, \quad \text{where} \quad (\widehat{\D}_p^{\pm})_k = (\D_p^{\pm})_k+\delta_{0,k}\CC C. \label{(26)}
\end{equation}

In order to apply the general results on quasifinite representations of $\ZZ$-graded Lie algebras developed in sect $2$ in \cite{kl}, we need to study the parabolic subalgebras of $\widehat{\D}_p^{\pm}$. Let us recall some general definition and results from \cite{kl}.

\medskip
Let $\g=\bigoplus_{j\in\ZZ}\g_j$, be a $\ZZ$-graded Lie algebra over $\CC$, and take $\g_{+}=\bigoplus_{j>0}\g_j$. A $\ZZ$-graded subalgebra $\p$ of $\g$ is called \textit{parabolic} if,
\[ \p =\bigoplus\limits_{j\in\ZZ} \p_j, \ \text{where} \ \p_j = \g_j \ \text{for} \ j \geq 0 \ \text{and} \ \p_j\neq0 \ \text{for some} \ j<0. \]

\newpage
\noindent We assume the following properties of $\g$:

\medskip
(P1) $\g_0$ is commutative,

\medskip
(P2) if $a\in\g_{-k} \ (k>0)$ and $[a,\g_1]=0$, then $a=0$.

\medskip
\noindent Given $a\in\g_{-1}$ that is nonzero, we define $\p^a=\bigoplus_{j\in\ZZ}\p_j^a$, where $\p_j^a=\g_j \ \text{for all} \ j\geq0$ and
\[ \p_{-1}^a=\sum[...[[a,\g_0],\g_0],...], \ \ \p_{-k-1}^a=[\p_{-1}^a,\p_{-k}^a]. \]

\noindent It was proved in \cite{kl} that $\p^a$ is the minimal parabolic subalgebra containing $a$.

\begin{defi}
\textit{(a)} A parabolic subalgebra $\p$ is called nondegenerate if $\p_{-j}$ has finite codimension in $\g_{-j}$ for all $j\in\NN$.

\textit{(b)} An element $a\in\g_{-1}$ is called nondegenerate if $\p^a$ is nonodegenerate.
\end{defi}
\noindent We will also require the following condition on $\g$:

\medskip
(P3) If $\p$ is a nondegenerate parabolic subalgebra of $\g$, then there exists a nondegenerate element $a\in\p_{-1}$.

\medskip
A $\g$-modulo $V$ is called $\ZZ$-\textit{graded} if $V=\bigoplus_{j\in\ZZ}V_j$ and $\g_iV_j\subset V_{i+j}$. A $\ZZ$-graded $\g$-module is called \textit{quasifinite} if $dimV_j<\infty$ for all j.

Given $\lambda\in\g_0^*$, a \textit{highest-weight module} is a $\ZZ$-graded $\g$-module $V(\g,\lambda)$ generated by a highest-weight vector $v_{\lambda}\in V(\g,\lambda)_0$ which satisfies
\[hv_{\lambda}=\lambda(h)v_{\lambda} \quad \text{for} \ h\in\g \ \text{and} \ \g_+v_{\lambda}=0.\]
A nonzero vector $v\in V(\g,\lambda)$ is called \textit{singular} if $\g_+v=0.$

The \textit{Verma module} over $\g$ is defined as usual:
\[ M(\g,\lambda)=\U(\g){\mbox{\small$\bigotimes$}}_{\U(\g_0\bigoplus\g_+)}\CC_{\lambda}, \]
where $\CC_{\lambda}$ is the one-dimensional ($\g_0\bigoplus\g_+$)-module given by $h\mapsto\lambda(h)$ if $h\in\g_0, \ \g_+\mapsto0$, and the action of $\g$ is induced by the left multiplication in $\U(\g)$. Here and further $\U(\g)$ stands for the universal enveloping algebra of the Lie algebra $\g$. Any highest weight module $V(\g,\lambda)$ is a quotient module of $M(\g,\lambda)$. The irreducible module $L(\g,\lambda)$ is the quotient of $M(\g,\lambda)$ by the maximal proper graded submodule.

Consider a parabolic subalgebra $\p =\bigoplus_{j\in\ZZ} \p_j$ of $\g$ and let $\lambda\in\g_0^*$ be such that $\lambda\mid_{\g_0\bigcap[\p,\p]}=0$. Then the ($\g_0\bigoplus\g_+$)-module $\CC_{\lambda}$ extends to a $\p$-module by letting $\p_j$ act as $0$ for $j<0$, and we may construct the highest weight module
\[ M(\p,\g,\lambda)=\U(\g){\mbox{\small$\bigotimes$}}_{\U(\p)}\CC_{\lambda} \]
called the \textit{generalized Verma module}. Clearly all these highest weight modules are graded. The following result givens the characterization of all irreducible quasifinite highest weight modules.

\newpage
\begin{teor} \label{lambda equivalencias}
\cite{kl} Let $\g=\bigoplus_{j\in\ZZ}\g_j$ be a $\ZZ$-graded Lie algebra over $\CC$ that satisfies conditions (P1), (P2) and (P3). The following conditions on $\lambda\in\g_0^*$ are equivalent:
 \begin{enumerate}[(a)]
  \item $M(\g,\lambda)$ contains a singular vector $a.v_{\lambda}$ in $M(\g,\lambda)_{-1}$, where a is nondegenerate.
  \item There exist a nonodegenerate element $a\in\g_{-1}$, such that $\lambda([\g_1,a])=0$.
  \item $L(\g,\lambda)$ is quasi-finite.
  \item There exist a nondegenerate element $a\in\g_{-1}$, such that $L(\g,\lambda)$ is the irreducible quotient of the generalized Verma module $M(\g,\p^a,\lambda)$.
 \end{enumerate}
\end{teor}

\begin{proof} See \cite{kl}. \end{proof}

Now we will prove that $\widehat{\D}_p^{\pm}$ satisfies the properties (P1), (P2), (P3) and therefore we can apply Theorem \ref{lambda equivalencias}. It is obvious that $\widehat{\D}_p^{\pm}$ satisfies (P1). In order to prove that $\widehat{\D}_p^{\pm}$ satisfies (P2) and (P3), we shall need the following results.
\begin{lema} \label{lema1}
Lets $f, \ g, \ h \in\CC[x]$ be such that $deg(fg)>0$ and
\begin{equation}
[t^kf(D),t^lg(D)]=t^{k+l}h(D)+\Psi(t^kf(D),t^lg(D))C, \label{(45)}
\end{equation}
then $deg(h)=deg(f)+deg(g)-1$ if and only if $deg(f)l\neq deg(g)k$.
\end{lema}
\begin{proof}
We suppose, $f(D)=D^i$ and $g(D)=D^j$ with $i+j\neq0$, then from \eqref{(45)}
\begin{equation*} h(D)=(D+l)^iD^j-(D+k)^jD^i, \end{equation*}
it is clear that $deg(h)\leqslant i+j-1$, moreover the coefficient of $D^{i+j-1}$ is $(il-jk)$ therefore, for this case the Lemma is true. Now, lets $f(x)=\sum_{i=0}^nf_ix^i$ and $g(x)=\sum_{j=0}^ng_jx^j$ be polynomials such that $n+m\neq0$, then
\begin{equation*}
[t^kf(D),t^lg(D)]=\sum_{i,j}f_ig_j[t^kD^i,t^lD^j]:=\sum_{i,j}f_ig_j\,t^{k+l}h_{i,j}(D),
\end{equation*}
with $deg(h_{i,j})\leqslant i+j-1$. Therefore we only have to study $t^{k+l}h_{n,m}(D):=[t^kD^n,t^lD^m]$, finally the proof follows from previous paragraph.
\end{proof}
\begin{lema} \label{codiemnsion de P} Let $\p=\bigoplus_{j\in\ZZ}\p_j$ be a $\ZZ$-graded subalgebra of $\widehat{\D}_p^{\pm}$ with $\p_0=(\widehat{\D}_p^{\pm})_0$.
 \begin{enumerate}[(a)]
   \item If $\p_j\neq0$, then it has finite codimension in $(\widehat{\D}_p^{\pm})_j$. \label{codimension de P, a}
   \item If $\p_{-1}\neq0$, then $\p_{-j}$ has finite codimension in $(\widehat{\D}_p^{\pm})_{-j}$ for all $j\in\NN$.\label{codimension de P, b}
 \end{enumerate}
\end{lema}

\begin{proof}
In order to prove \eqref{codimension de P, a}, it is enough to find a family $\{t^jg_k(D)p(D)\}_{k\geqslant1}\subset\p_j$ with $deg(g_k)=m_0+2k$ for some fixed $m_0\in\ZZ_+$ (see Lemma \ref{determino Dp+-}).

We suppose $j\neq0$. Let $t^jf(D)p(D)\in\p_j$ be nonzero. By hypothesis and Lemma \ref{determino Dp+-}, $(D-\dfrac{c}{2})^{2k+\overline{n}}p(D)\in\p_0$ for all $k\geqslant1$, then
\begin{equation*} [t^jf(D)p(D),(D-\dfrac{c}{2})^{2k+\overline{n}}p(D)]:=t^jg_k(D)p(D)\in\p_j, \end{equation*}
and by Lemma \ref{lema1} we obtain $deg(g_k)=deg(f)+\overline{n}+n-1+2k$, finishing (a).

Now, in order to prove \eqref{codimension de P, b} we only need to see that $\p_{-j}\neq0$ for all $j\geqslant1$. By induction, we suppose $\p_{-j}\neq0$ with $j\geqslant1$. Then from the above paragraph, for all $k\geqslant1$ there exists $t^{-j}g_k(D)p(D)\in\p_{-j}$ with $deg(g_k)=m_0+2k$ ($m_0\in\ZZ_+$ fixed) and by hypothesis there exists $t^{-1}f(D)p(D)\in\p_{-1}$ that is nonzero. Hence, we can take $k_0\in\NN$ such that $(n+deg(f))j\neq(n+m_0+2k_0)$, then by Lemma \ref{lema1}, we have that $[t^{-1}f(D)p(D),t^{-j}g_{k_0}(D)p(D)]\in\p_{-j-1}$ is nonzero.
\end{proof}

\begin{corol}
\begin{enumerate}[(a)]
\item $\widehat{\D}_p^{\pm}$ satisfies (P2).
\item Any parabolic subalgebra of $\widehat{\D}_p^{\pm}$ is nondegenerate. \label{(46)}
\item Any nonzero element of $(\widehat{\D}_p^{\pm})_{-1}$ is nondegenerate. \label{(47)}
\item $\widehat{\D}_p^{\pm}$ satisfies (P3). \label{(48)}
\end{enumerate}
\end{corol}
\begin{proof}
Let $t^{-k}f(D)p(D)\in\widehat{\D}_p^{\pm}$ be nonzero (with $k>0$), then if we take $tg(D)p(D)\in\widehat{\D}_p^{\pm}$ with no constant $g$, we obtain from Lemma \ref{lema1} that
\begin{equation*} [t^{-k}f(D)p(D),tg(D)p(D)]\neq 0, \end{equation*}
therefore, $\widehat{\D}_p^{\pm}$ satisfies (P2).

Now, let $\p$ be a parabolic subalgebra of $\widehat{\D}_p^{\pm}$, by definition there exists $j\in\NN$ such that $\p_{-j}\neq0$ then by (P2), $\p_{-1}\neq0$, and the proof of \eqref{(46)} follows from Lemma \ref{codiemnsion de P} \eqref{codimension de P, b}. Finally, \eqref{(47)} follows from \eqref{(46)}, and \eqref{(48)} follows from \eqref{(47)}.
\end{proof}

Let $L(\widehat{\D}_p^{\pm},\lambda)$ be an irreducible quasifinite highest weight module over $\widehat{\D}_p^{\pm}$. By Theorem \ref{lambda equivalencias}, there exists some monic polynomial $b(x-\frac{c+1}{2})p(x)$ such that $(t^{-1}b(D-\frac{1+c}{2})p(D))v_{\lambda}=0$ (with $b(x)$ a polynomial odd or even depending on $\widehat{\D}_p^{\pm}$ as described in Lemma \ref{determino Dp+-}). We shall call such monic polynomial of minimal degree, uniquely determined by the highest-weight $\lambda$, the \textit{characteristic polynomial} of $L(\widehat{\D}_p^{\pm},\lambda)$.

Let us denote by $\ZZ_+^{(0)}$ (resp. $\ZZ_+^{(1)}$) the set all even (resp. odd) non-negative integers.

A functional $\lambda\in(\widehat{\D}_p^{\pm})_0^*$ is described by its labels $\triangle_l=-\lambda((D-\frac{c}{2})^lp(D))$, where $l\in\ZZ_+^{(\overline{n})}, \ n=deg(p)$ and the central charge $\lambda(C)=c_0$. We can consider the generating series
\begin{eqnarray}
\Delta_{\lambda}(x)=\sum_{l\in\ZZ_+^{(\overline{n})}}\dfrac{x^l}{l!}\triangle_l. \label{(31)}
\end{eqnarray}

Recall that a \textit{quasipolynomial} is a linear combination of functions of the form $q(x)e^{\alpha x}$, where $q\in\CC[x]$ and $\alpha\in\CC$. Also we have a well-known characterization: a formal power series is a  quasipolynomial (resp. even quasipolynomial) if and only if it satisfies a non-trivial linear differential equation with constant coefficients $f(\partial)=0$, where $f(x)$ is an  polynomial (resp. even polynomial).

The following theorem characterizes the irreducible quasifinite highest weight modules over $\widehat{\D}_p^{\pm}$.

\begin{teor} \label{caractrizacion de MWH}
A $\widehat{\D}_p^{\pm}$-module $L(\widehat{\D}_p^{\pm},\lambda)$ is quasifinite if and only if
\begin{equation}
\triangle_{\lambda}(x)=p\left(\frac{d}{dx}+\frac{c}{2}\right)\left(\frac{\phi_{\lambda}(x)}{2\sinh\left(\frac{x}{2}\right)}\right), \label{(30)}
\end{equation}
where $\phi_{\lambda}(x)$ is an even quasipolynomial such that $\phi_{\lambda}(0)=0$.
\end{teor}

\begin{proof}
Recall that $p$ satisfies
\begin{equation} p(x)=(-1)^np(-x+c) \label{(p(x)=(-1)^np(-x+c))} \end{equation}
where $n=deg(p)$. We shall use the following identities, for $f, \ g\in\CC[x]$ and $a\in\CC$:
\begin{eqnarray}
f\left(\pm\frac{d}{dx}\right)e^{ax}&=&f(\pm a)e^{ax}, \label{(32)} \\
e^{\pm x(D-\frac{c}{2})}f(D)&=&f\left(\pm \frac{d}{dx}+\frac{c}{2}\right)e^{\pm x(D-\frac{c}{2})}, \label{(23)} \\
e^{\pm \frac{x}{2}}f\left(\frac{d}{dx}\right)g(x)&=&f\left(\frac{d}{dx}\mp\frac{1}{2}\right)e^{\pm\frac{x}{2}}g(x). \label{(25)}
\end{eqnarray}
Using \eqref{(p(x)=(-1)^np(-x+c))} and \eqref{(23)},
\begin{eqnarray}
\Delta_{\lambda}(x)&=&-\frac{1}{2}\,\lambda\left(\left(e^{x(D-\frac{c}{2})}+(-1)^{n+1}e^{-x(D-\frac{c}{2})} \right)p(D)\right) \notag \\
                      &=&-\frac{1}{2}\,\lambda\left(p\left(\frac{d}{dx}+\frac{c}{2}\right)\left(e^{x(D-\frac{c}{2})}-e^{-x(D-\frac{c}{2})} \right)\right). \label{(11)}
\end{eqnarray}
Now, take $\Gamma_{\lambda}(x)$ a solution of
\begin{equation} \triangle_{\lambda}(x)= p\left(\frac{d}{dx}+\frac{c}{2}\right)\Gamma_{\lambda}(x). \label{(12)} \end{equation}

It follows from Theorem \ref{lambda equivalencias} that $L(\widehat{\D}_p^{\pm},\lambda)$ is quasifinite if and only if there exists $t^{-1}b(D-\frac{c+1}{2})p(D)\in(\widehat{\D}_p^{\pm})_{-1}$ such that
\begin{equation} 0=\lambda([t(D-\frac{c-1}{2})^{2k+\delta}p(D),t^{-1}b(D-\frac{c+1}{2})p(D)])  \label{(24)} \end{equation}
for all $k\in\ZZ$ and by Lemma \ref{determino Dp+-} and \eqref{(26)}, $b\in\CC[x]^{(\delta)}$ where $\delta$ is given by
\begin{equation}
 \delta =\begin{cases}
             \overline{n}, \ \ \text{in the case} \ \widehat{\D}_p^+, \\
             \overline{n-1}, \ \ \text{in the case} \ \widehat{\D}_p^-.
            \label{(27)} \end{cases}
\end{equation}
Taking generatriz serie, \eqref{(24)} is equivalent to
\begin{eqnarray*}
0&=&\frac{1}{2}\lambda([t(e^{x(D-\frac{c-1}{2})}+(-1)^{\delta}e^{-x(D-\frac{c-1}{2})})p(D),t^{-1}b(D-\frac{c+1}{2})p(D)]) \\
&=&\frac{1}{2}\lambda\left(b(D-\frac{c+1}{2})p(D-1)p(D)(e^{x(D-\frac{c+1}{2})}+(-1)^{\delta}e^{-x(D-\frac{c+1}{2})}) \right. \\
& &-b(D-\frac{c-1}{2})p(D+1)p(D)(e^{x(D-\frac{c-1}{2})}+(-1)^{\delta}e^{-x(D-\frac{c-1}{2})}) \\
& &\left. +b(-\frac{c+1}{2})p(-1)p(0)(e^{-(\frac{c+1}{2})x}+(-1)^{\delta}e^{(\frac{c+1}{2})x})C \right)
\end{eqnarray*}
Then using the identities \eqref{(p(x)=(-1)^np(-x+c))}, \eqref{(32)}, \eqref{(23)}, \eqref{(25)}, \eqref{(11)} and \eqref{(12)}
\begin{eqnarray*}
0&=&\frac{1}{2}\lambda\left(b(\frac{d}{dx})p(\frac{d}{dx}+\frac{c-1}{2})e^{-\frac{x}{2}}p(\frac{d}{dx}+\frac{c}{2})e^{x(D-\frac{c}{2})} \right. \\
& &+(-1)^{\delta}b(-\frac{d}{dx})p(-\frac{d}{dx}+\frac{c-1}{2})e^{\frac{x}{2}}p(-\frac{d}{dx}+\frac{c}{2}+1)e^{-x(D-\frac{c}{2})} \\
& &-b(\frac{d}{dx})p(\frac{d}{dx}+\frac{c+1}{2})e^{\frac{x}{2}}p(\frac{d}{dx}+\frac{c}{2})e^{x(D-\frac{c}{2})} \\
& &\left. -(-1)^{\delta}b(-\frac{d}{dx})p(-\frac{d}{dx}+\frac{c+1}{2})e^{-\frac{x}{2}}p(-\frac{d}{dx}+\frac{c}{2}-1)e^{-x(D-\frac{c}{2})} \right) \\
& &+\frac{1}{2}\left(b(\frac{d}{dx})p(\frac{d}{dx}+\frac{c-1}{2})p(\frac{d}{dx}+\frac{c+1}{2})e^{-(\frac{c+1}{2})x}
\right. \\
& & \left.
+(-1)^{\delta}b(-\frac{d}{dx})p(-\frac{d}{dx}+\frac{c-1}{2})p(-\frac{d}{dx}+\frac{c+1}{2})e^{(\frac{c+1}{2})x}\right)c_0
\\
&=&\frac{1}{2}\lambda\left(b(\frac{d}{dx})\left[p(\frac{d}{dx}+\frac{c-1}{2})e^{-\frac{x}{2}}-p(\frac{d}{dx}+\frac{c+1}{2})e^{\frac{x}{2}}\right]p(\frac{d}{dx}+\frac{c}{2})e^{x(D-\frac{c}{2})}+
\right. \\
& &\left. b(\frac{d}{dx})\left[p(\frac{d}{dx}+\frac{c+1}{2})e^{\frac{x}{2}}-p(\frac{d}{dx}+\frac{c-1}{2})e^{-\frac{x}{2}}\right]p(\frac{d}{dx}+\frac{c}{2})e^{-x(D-\frac{c}{2})} \right)
\\
& &+\frac{1}{2}b(\frac{d}{dx})p(\frac{d}{dx}+\frac{c-1}{2})p(\frac{d}{dx}+\frac{c+1}{2})(e^{-(\frac{c+1}{2})x}+e^{(\frac{c+1}{2})x})c_0
\\
&=&b(\frac{d}{dx})\left[p(\frac{d}{dx}+\frac{c+1}{2})e^{\frac{x}{2}}-p(\frac{d}{dx}+\frac{c-1}{2})e^{-\frac{x}{2}}\right] p(\frac{d}{dx}+\frac{c}{2})\Gamma_{\lambda}(x)\\
& &+b(\frac{d}{dx})p(\frac{d}{dx}+\frac{c-1}{2})p(\frac{d}{dx}+\frac{c+1}{2})\cosh((\frac{c+1}{2})x)c_0
\\
&=&b(\frac{d}{dx})p(\frac{d}{dx}+\frac{c+1}{2})p(\frac{d}{dx}+\frac{c-1}{2})\left(2 \sinh(\frac{x}{2})\Gamma_{\lambda}(x)+\cosh((\frac{c+1}{2})x)c_0\right).
\end{eqnarray*}
It follows that $L(\widehat{\D}_p^{\pm},\lambda)$ is quasifinite if and only if there exists $b\in\CC[x]^{\delta}$ (see \eqref{(27)}) such that
\begin{equation}
0=b(\frac{d}{dx})p(\frac{d}{dx}+\frac{c+1}{2})p(\frac{d}{dx}+\frac{c-1}{2})\left(2 \sinh(\frac{x}{2})\Gamma_{\lambda}(x)+\cosh((\frac{c+1}{2})x)c_0\right). \label{(13)}
\end{equation}
Therefore, if $L(\widehat{\D}_p^{\pm},\lambda)$ is quasifinite, then $2 \sinh(\frac{x}{2})\Gamma_{\lambda}(x)+\cosh((\frac{c+1}{2})x)c_0$ is a quasipolynomial. But, using \eqref{(11)} and \eqref{(12)}, we get that $\Gamma_{\lambda}(x)$ is an odd function. Hence,
\begin{equation} \phi_{\lambda}(x)=2 \sinh(\frac{x}{2})\Gamma_{\lambda}(x) \label{(36)} \end{equation}
is an even quasipolynomial such that $\phi_{\lambda}(0)=0$, and using \eqref{(12)}, we have
\begin{equation}
\Delta_{\lambda}(x)=p\left(\frac{d}{dx}+\frac{c}{2}\right)\left(\frac{\phi_{\lambda}(x)}{2\sinh(\frac{x}{2})}\right). \label{(28)}
\end{equation}
Conversely, if \eqref{(28)} holds for some even quasipolynomial $\phi_{\lambda}$ with $\phi_{\lambda}(0)=0$, then $F(x)=\phi_{\lambda}(x)+\cosh((\frac{c+1}{2})x)c_0$ is an even quasipolynomial and it satisfies $q(\frac{d}{dx})F(x)=0$ for some $q\in\CC[x]^{\delta}$. In particular, we have
\[p(\frac{d}{dx}+\frac{c+1}{2})p(\frac{d}{dx}+\frac{c-1}{2})q(\frac{d}{dx})F(x)=0,\]
and therefore $L(\widehat{\D}_p^{\pm},\lambda)$ is quasifinite, finishing the proof.
\end{proof}

The even quasipolynomial $\phi_{\lambda}(x)+\cosh((\frac{c+1}{2})x)c_0$, where $\phi_{\lambda}(x)$ is from \eqref{(30)} and $c_0$ is the central charge, can be written in the form
\begin{equation}
\phi_{\lambda}(x)+\cosh((\frac{c+1}{2})x)c_0=\sum_i q_i(x)\cosh(e_i^+x) + \sum_j r_j(x)\sinh(e_j^-x), \label{(29)}
\end{equation}
where $q_i(x)$ (resp. $r_j(x)$) are non-zero even (resp. odd) polynomials and $e_i^+$ (resp. $e_j^-$) are distinct complex numbers. Note that $\sum_iq_i(0)=c_0$.

The expression \eqref{(29)} is unique up to a sign of $e_i^+$ or a simultaneous change of signs of $e_j^-$ and $r_j(x)$. We call $e_i^+$ (resp. $e_j^-$) the \textit{even type} (resp. \textit{odd type}) \textit{exponents} of $L(\widehat{\D}_p^{\pm},\lambda)$ with \textit{multiplicities} $q_i(x)$ (resp. $r_j(x)$). We denote by $e^+$ the set of even type exponents $e_i^+$ with multiplicity $q_i(x)$ and by $e^-$ the set of odd type exponents $e_j^-$ with multiplicity $r_j(x)$. Then the pair $(e^+,e^-)$ determines $L(\widehat{\D}_p^{\pm},\lambda)$ uniquely, and we shall also denote it as $L(\widehat{\D}_p^{\pm}; e^+,e^-)$.

\begin{corol}
Let $L(\widehat{\D}_p^{\pm},\lambda)$ be an irreducible quasifinite highest weight module over $\widehat{\D}_p^{\pm}$, $b(x-\frac{c+1}{2})p(x)$ be its characteristic polynomial with $b(x)\in\CC[x]^{\delta}$ (see \eqref{(27)}), $\Gamma_{\lambda}(x)$ be a solution of \eqref{(12)} and let $F(x)=2\sinh(\frac{x}{2})\Gamma_{\lambda}(x)+\cosh((\frac{c+1}{2})x)c_0$. Then
\begin{equation*}
b\left(\frac{d}{dx}\right)p\left(\frac{d}{dx}+\frac{c+1}{2}\right)
p\left(\frac{d}{dx}+\frac{c-1}{2}\right)F(x)=0
\end{equation*}
is the minimal order homogeneous linear differential equation with constant coefficients of the form
\begin{equation*} f\left(\frac{d}{dx}\right)p\left(\frac{d}{dx}+\frac{c+1}{2}\right)
p\left(\frac{d}{dx}+\frac{c-1}{2}\right) \ \text{with} \ f\in\CC[x]^{\delta}, \end{equation*}
satisfied by $F(x)$. Moreover, the exponents appearing in \eqref{(29)} are all roots of the polynomial $b(x)p(x+\frac{c+1}{2})p(x+\frac{c-1}{2})$.
\end{corol}

\section{Interplay between $\widehat{\D}_x^{\ho, \, \pm}$ and $\gl$, $c_{\infty}^{[m]}$, $d_{\infty}^{[m]}$ and $\mathcal{L}^{[m]}_{\pm}$}

Denote by $\Rm$ the quotient algebra $\CC[u]/(u^{m+1})$ and by $\textbf{1}$ the identity element of $\Rm$. 
We let $\glm$ the Lie algebra of all matrices $(a_{ij})_{i,j\in\ZZ}$ with finitely many nonzero diagonals with entries in $\Rm$. Also denote by $E_{ij}$ the infinite matrix with $\textbf{1}$ at $(i,j)$ place and $0$ elsewhere. Letting $wt \, E_{ij}=j-i$ defines the \textit{principal $\ZZ$-gradation} of $\glm$. There is a natural automorphism $\nu$ of $\glm$ given by
\begin{equation} \nu(E_{i, \, j})=E_{i+1, \, j+1}. \label{(14)} \end{equation}
Consider the following two-cocycle on $\glm$ with values in $\Rm$;
\begin{equation} C(A,B)=tr([J,A]B), \label{(15)} \end{equation}
where $J=\sum_{j\leqslant0}E_{jj}$, and denote by $\gl=\glm\oplus\Rm$ the corresponding central extension. The $\ZZ$-gradation of this Lie algebra extends from $\glm$ by letting $wt \,\Rm=0$.

Given $\lambda \in (\gl)_0^*$, we let
\begin{eqnarray}
c_j&=&\lambda(u^j), \notag \\
^a \lambda_i^{(j)}&=&\lambda(u^jE_{ii}), \label{(40)} \\
^ah_i^{(j)}&=&^a \lambda_i^{(j)}- ^a \lambda_{i+1}^{(j)}+ \delta_{i,0}\,c_j. \notag
\end{eqnarray}
where $i\in \ZZ$ and $j=0,\dots ,m$. The superscript $a$ corresponds to the type A Lie algebra $\gl$. Let $L(\gl,\lambda)$ be the irreducible highest weight $\gl$-module with highest weight $\lambda$. The $^a \lambda_i^{(j)}$ are called the {\it labels} and $c_j$ are the {\it central charges} of $L(\gl,\lambda)$.

Consider the vector space $\Rm[t,t^{-1}]$ and take its basis $v_i=t^i \ (i\in\ZZ)$ over $\Rm$. Let us consider the following $\CC$-bilinear forms on this space:
\begin{eqnarray*}
C(u^mv_i,u^nv_j)&=&u^m(-u)^n(-1)^i\delta_{i,1-j}, \\
D(u^mv_i,u^nv_j)&=&u^m(-u)^n\delta_{i,1-j}.
\end{eqnarray*}
Denote by $\overline{c}_{\infty}^{[m]}$ and $\overline{d}_{\infty}^{[m]}$ the Lie subalgebras of $\glm$ which preserves the bilinear forms $C$ and $D$ respectively. We have;
\begin{eqnarray*}
\overline{c}_{\infty}^{[m]}&=&\{A\in\glm \, | \, A_{i,j}(u)=(-1)^{i+j+1}A_{1-j,1-i}(-u)\}, \\
\overline{d}_{\infty}^{[m]}&=&\{ A\in\glm \, | \, A_{i,j}(u)=-A_{1-j,1-i}(-u) \}.
\end{eqnarray*}
Denote by $c_{\infty}^{[m]}=\overline{c}_{\infty}^{[m]}{\mbox{\small$\oplus$}}\Rm$ and $d_{\infty}^{[m]}=\overline{d}_{\infty}^{[m]}\oplus R_m$ the central extension of $\overline{c}_{\infty}^{[m]}$ and $\overline{d}_{\infty}^{[m]}$ respectively, given by the restriction of the two-cocycle \eqref{(15)}. This subalgebras inherits from $\gl$ the principal $\ZZ$-gradation.

Let $\g$ stand for $c_{\infty}^{[m]}$ or $d_{\infty}^{[m]}$. Given $\lambda \in (\g)_0^*$, denote by $L(\g ,\lambda)$ the irreducible highest weight module over $\g$ with highest weight $\lambda$. We let
\begin{eqnarray}
c_j&=&\lambda(u^j), \notag \\
^g \lambda_i^{(j)}&=&\lambda(u^jE_{ii}-(-u)^jE_{1-i,1-i}), \notag \\
^g h_i^{(j)}&=&^{\g} \lambda_i^{(j)}- ^{\g}\lambda_{i+1}^{(j)}, \label{(41)} \\
^g h_0^{(j)}&=&^{\g} \lambda_1^{(j)}+c_j \ \ \ \text{(j even)}, \notag
\end{eqnarray}
where $i\in \ZZ$, $j=0,\dots ,m$ and the superscript  $g$ represents $c$ or $d$ depending on whether $\g$ is $c_{\infty}^{[m]}$ or $d_{\infty}^{[m]}$.  The $^g \lambda_i^{(j)}$ are called the {\it labels} and $c_j$ are the {\it central charges} of $L(\g,\lambda)$.

We define
\begin{eqnarray*}
\overline{\li}^{[m]}_{\pm}&=&\{ A\in\glm \, | \, A_{i,j}(u)=-(\mp1)^{i+j}A_{-j,-i}(-u) \ \text{if} \ ij>0 \ \vee \ i=j=0; \\
                  & & A_{i,j}(u)=(\mp1)^{i+j}A_{-j,-i}(-u) \ \text{if} \ ij<0; \ A_{-i,0}(u)=(\mp1)^iuA_{0,i}(-u); \\
                  & & A_{i,0}(u)=-(\mp1)^iuA_{0,-i}(-u) \ \text{for} \ i\in\NN \},
\end{eqnarray*}
subalgebras of $\glm$. Denote by $\li^{[m]}_{\pm}=\overline{\li}^{[m]}_{\pm}{\mbox{\small$\oplus$}}\Rm$ the central extension of $\overline{\li}^{[m]}_{\pm}$, given by the restriction of the two-cocycle \eqref{(15)}. This subalgebras inherits from $\gl$ the principal $\ZZ$-gradation.

Given $\lambda \in (\mathcal{L}^{[m]}_{\pm})_0^*$, we let
\begin{eqnarray}
c_j&=&\lambda(u^j), \notag \\
^{\pm} \lambda_i^{(j)}&=&\lambda(u^jE_{i,i}-(-u)^jE_{-i,-i}), \label{(44)} \\
^{\pm} h_i^{(j)}&=&^{\pm} \lambda_i^{(j)}-^{\pm} \lambda_{i+1}^{(j)}+\delta_{i,0}\,c_j. \notag
\end{eqnarray}
where $i\in \ZZ$ and $j=0,\dots ,m$. The superscript $\pm$ corresponds to the  Lie algebras $\li^{[m]}_{\pm}$. Let $L(\li^{[m]}_{\pm},\lambda)$ be the irreducible highest weight $\gl$-module with highest weight $\lambda$. The $^{\pm}\lambda_i^{(j)}$ are called the {\it labels} and $c_j$ are the {\it central charges} of $L(\li^{[m]}_{\pm},\lambda)$.

\

Let $\ho$ denote the algebra of all holomorphic function on $\CC$ with the topology of uniform convergence on compact sets and $\ho^{(1)}$ (resp. $\ho^{(0)}$) the set of odd (resp. even) holomorphic function. We consider the vector space $\D^{\ho}$ spanned by the differential operators (of infinite order) of the form $t^kf(D)$, where $f\in\ho$. The bracket in $\D$ naturally extends to $\D^{\ho}$. Similarly, we define a completion $\D_x^{\ho,\,-}$ (resp. $\D_x^{\ho,\,+}$) of $\D_x^{-}$ (resp. $\D_x^{+}$) consisting of all differential operators of the form $t^kf(D+\frac{k}{2})D$ where $f\in\ho^{(\overline{k})}$ (resp. $f\in\ho^{(0)}$).

Then the two-cocycle $\Psi$ on $\D$ (resp. $\D_x^{\pm}$) extends to a two-cocycle on $\D^{\ho}$ (resp. $\D_x^{\ho,\,\pm}$). We denote the corresponding central extension by $\widehat{\D}^{\ho}=\D^{\ho}{\mbox{\small$\oplus$}}\CC\mathbf{C}$ (resp. $\widehat{\D}_x^{\ho,\,\pm}=\D_x^{\ho,\,\pm}{\mbox{\small$\oplus$}}\CC\mathbf{C}$).

Given $s\in\CC$, we will consider a family of homomorphism of Lie algebras $\varphi_s^{[m],\,\pm}:\D_x^{\ho,\,\pm}\rightarrow\glm$ defined by;
\begin{eqnarray}
 \varphi_s^{[m],\,\pm}(t^kf(D+\frac{k}{2})D)&=&\sum_{j\in\ZZ}f(-j+\frac{k}{2}+s+u)(-j+s+u)E_{j-k,j} \notag \\
                           &=&\sum_{j\in\ZZ}\sum_{i=0}^m \frac{(f(-j+k/2+s)(-j+s))^{(i)}}{i!}u^i E_{j-k,j} \notag \\
                           &=&\sum_{j\in\ZZ}\sum_{i=0}^m \frac{f^{(i)}(-j+\frac{k}{2}+s)}{i!}((-j+s)+u)u^iE_{j-k,j}, \notag \\
                           & & \label{fis}
\end{eqnarray}
where $f^{(i)}$ denotes the \emph{i}th derivative. Note that $\varphi_s^{[m],\,\pm}$ is the restriction to $\D_x^{\ho,\,\pm}$ of the homomorphism (3.2.1) in Ref. \cite{kr1}.
\begin{obse} The principal $\ZZ$-gradations on $\D_x^{\ho,\,\pm}$ and $\glm$ are compatible under the homomorphisms $\varphi_s^{[m], \, \pm}$. \end{obse}

Let
\begin{eqnarray*}
I_{s,\,k}^{[m],\,-}&=&\{f\in\ho^{(\overline{k})}: f^{(i)}(-j+k/2+s)=0 \ \text{for all} \ j\in\ZZ, \ 0\leqslant i\leqslant m\}, \\
I_{s,\,k}^{[m],\,+}&=&\{f\in\ho^{(0)}: f^{(i)}(-j+k/2+s)=0 \ \text{for all} \ j\in\ZZ, \ 0\leqslant i\leqslant m\},
\end{eqnarray*}
and let
\begin{eqnarray*}
J_s^{[m],\,-}&=&\bigoplus_{k\in\ZZ}\{t^kf(D+k/2):f\in I_{s,\,k}^{[m],\,-}\}, \\
J_s^{[m],\,+}&=&\bigoplus_{k\in\ZZ}\{t^kf(D+k/2):f\in I_{s,\,k}^{[m],\,+}\}.
\end{eqnarray*}

\begin{prop} \label{prop D and gl} Given $s\in(\CC-\ZZ/2)$ and $m\in\ZZ_+$ we have the following exact sequence of Lie algebras:
\begin{equation*}
 0\rightarrow J_s^{[m],\,\pm}\rightarrow \D_x^{\ho,\,\pm} \xrightarrow{\varphi_s^{[m],\pm}} \mathfrak{gl}_{\infty}^{[m]}\rightarrow 0.
 \end{equation*}
\end{prop}

\begin{proof}
It is clear that $ker \varphi_s^{[m],\,\pm}=J_s^{[m],\,\pm}$. We only  need to prove that $\varphi_s^{[m],\pm}$ is surjective. We recall the following well known fact: for every discrete sequence of points in $\CC$ and a non-negative integer $m$ there exists $a(x)\in\ho$ having the prescribed values of its first $m$ derivatives at these points.

Case $\varphi_s^{[\overrightarrow{m}],-}$ (resp. $\varphi_s^{[\overrightarrow{m}],+}$):  since $s\not\in\ZZ/2$ the sequence $\{-j+k/2+s\}_{j\in\ZZ}$ and $\{j-k/2-s\}_{j\in\ZZ}$ are disjoint. We fix $0\leqslant i_0\leqslant m$ and $j_0, \ k\in\ZZ$, then there exists $a(x)\in\ho$ such that

\medskip
$\cdot$ if $k$ is even, $a^{(i)}(-j+k/2+s)=a^{(i)}(j-k/2-s)=\delta_{i,\,i_0}\delta_{j,\,j_0}/2$,

\medskip
$\cdot$ if $k$ is odd, $a^{(i)}(-j+k/2+s)=2\delta_{i,\,i_0}\delta_{j,\,j_0}, \ \ a^{(i)}(j-k/2-s)=\delta_{i,\,i_0}\delta_{j,\,j_0}$,

\medskip
\noindent (resp. $a^{(i)}(-j+k/2+s)=a^{(i)}(j-k/2-s)=\delta_{i,\,i_0}\delta_{j,\,j_0}/2$), and let
\begin{equation*} g(x)=i_0!(a(x)+(-1)^ka(-x)), \end{equation*}
(resp. $g(x)=i_0!(a(x)+a(-x))$). Then
\begin{equation*}
\varphi_s^{[m],\pm}(t^kg(D+k/2)D)=(u+(s-j_0))u^{i_0}E_{j_0-k,\,j_0},
\end{equation*}
and since $\{(u+(s-j_0))u^i\}_{i=0}^m$ is basis of $\Rm$, then  \ $\varphi_s^{[m],\pm}$ is surjective, finishing the proof.
\end{proof}

\begin{prop} \label{prop D and c, d}
For $s=\frac{1}{2}$, we have the following exact sequence of Lie algebras:
\begin{equation*}
0 \rightarrow J^{[m]}_{\frac{1}{2}} \rightarrow \D_x^{\ho,\,\pm} \rightarrow \overline{\g}_{\pm} \rightarrow 0,
\end{equation*}
where $\overline{\g}_+:= \overline{c}_{\infty}^{[m]}$ and $\overline{\g}_-:= \overline{d}_{\infty}^{[m]}$.
\end{prop}

\begin{proof}
The homomorphism $\varphi_{\frac{1}{2}}^{[m]}: \D_x^{\ho} \rightarrow \glm$ defined by
\begin{eqnarray*}
\varphi_{\frac{1}{2}}^{[m]}(t^kf(D)D)&=&\sum_{j\in\ZZ}f(u+\frac{1}{2}-j)(u+\frac{1}{2}-j)E_{j-k,j} \\
                                     &=&\sum_{j\in\ZZ}\sum_{i=0}^m \frac{f^{(i)}(\frac{1}{2}-j)}{i!}((\frac{1}{2}-j)+u)u^iE_{j-k,j},
\end{eqnarray*}
is surjective and the anti-involution $\sigma_{\pm}$ is transferred through this homomorphism  to an anti-involution $w_{\pm}$ in $\glm$, that satisfies
\begin{equation*}
w_{\pm}\left( (u + \frac{1}{2}-j)f(u) E_{i,j} \right)=(\pm 1)^{i+j}(-u + \frac{1}{2}-i)f(-u) E_{1-j,1-i},
\end{equation*}
with $f\in\CC[x]$, from which it is easy to see that 
\begin{equation}
w_{\pm}(f(u) E_{i,j})=(\pm 1)^{i+j}(-u + \frac{1}{2}-i)(-u + \frac{1}{2}-j)^{-1}f(-u) E_{1-j,1-i}. \label{(38)}
\end{equation}
Then, the Lie algebra of $-\sigma_{\pm}$-fixed points in $\D_x^{\ho}$ (namely $\D_x^{\ho,\,\pm}$), maps surjectively to the Lie algebra of $-w_{\pm}$-fixed points in $\glm$.

Now, we define the automorphism $T:\glm \longrightarrow \glm$ by
\begin{eqnarray*}
T(u^l E_{i,j})&=&\prod_{k=i}^{j-1} \left(u-(k+\frac{1}{2})\right)u^lE_{i,j} \ \ \ \text{if $i<j$}, \\
T(u^l E_{i,i})&=&u^l E_{i,i}, \\
T(u^l E_{i,j})&=&\prod_{k=j}^{i-1} \left(u-(k+\frac{1}{2})\right)^{-1}u^lE_{i,j} \ \ \ \text{if $i>j$}.
\end{eqnarray*}
On the other hand, let $\rho_{\pm}(f(u) E_{i,j})=(\mp1)^{i+j}f(-u) E_{1-j,1-i}$ be the anti-involution in $\glm$ that define $\overline{\g}_{\pm}$, then using \eqref{(38)} we have  that
\begin{equation} T \rho_{\pm} |_{(\glm)_{- 1} \bigoplus (\glm)_0\bigoplus (\glm)_1} = w_{\pm} T |_{(\glm)_{- 1} \bigoplus (\glm)_0\bigoplus (\glm)_1}. \label{(39)} \end{equation}
Then, since $\glm$ is generated by $ (\glm)_{-1} \bigoplus (\glm)_0 \bigoplus (\glm)_1$ and using  \eqref{(39)}, we obtain that $\rho_{\pm} = T ^{-1} w_{\pm} T$. As before, the Lie algebra of $-w_{\pm}$-fixed points in $\glm$ maps surjectively through of $T^{-1}$ to the Lie algebra of $-\rho_{\pm}$-fixed points in $\glm$, namely $\overline{\g}_{\pm}$. So $T^{-1} \, \varphi_{\frac{1}{2}}^{[m]}$ maps surjectively $\D_x^{\ho,\,\pm}$ in $\overline{\g}_{\pm}$ and since $T^{-1}$ is an automorphism it is clear that $ker T^{-1} \, \varphi_{\frac{1}{2}}^{[m],\,\pm}=J_{\frac{1}{2}}^{[m],\,\pm}$, finishing the proof.
\end{proof}

\begin{prop} \label{prop D and L}
For $s=0$, we have the following exact sequence of Lie algebras:
\begin{equation*}
0 \rightarrow J^{[m]}_{0} \rightarrow \D_x^{\ho,\,\pm} \rightarrow \overline{\li}^{[m]}_{\pm} \rightarrow 0.
\end{equation*}
\end{prop}

\begin{proof}
We consider the following morphisms of Lie algebras: $\varphi_0^{[m]}$ as in \eqref{fis} and $T:\glm \longrightarrow \glm$ defined by
\begin{eqnarray*}
T(u^l E_{i,j})&=&\prod_{\substack {k=i \\ k\neq0}}^{j-1} (u-k)\,u^lE_{i,j} \ \ \ \text{if $i<j$}, \\
T(u^l E_{i,i})&=&u^l E_{i,i}, \\
T(u^l E_{i,j})&=&\prod_{\substack {k=i \\ k\neq0}}^{i-1} (u-k)^{-1}u^lE_{i,j} \ \ \ \text{if $i>j$},
\end{eqnarray*}
note that $T(u^l E_{i,j})=u^l \,T(E_{i,j})$. Then, $T\varphi_0^{[m]}:\D_x^{\ho,\,\pm}\rightarrow \overline{\li}^{[m]}_{\pm}$ is surjective and since $T$ is an automorphism, it is clear that $ker T \varphi_0^{[m]}=J_0^{[m],\,\pm}$, finishing the proof.
\end{proof}

\begin{obse} \label{obse1}
\textit{(a)} For $ s =0$ and $s=\frac{1}{2} $, by an abuse the notation we will denote again $\varphi_s^{[m], \, \pm}$ the surjective homomorphism of the Propositions \ref{prop D and c, d}, \ref{prop D and L}.

\noindent \textit{(b)} For $s\in\ZZ$ (respectively $s\in\ZZ+\frac{1}{2}$) the image of $\D_x^{\ho,\pm}$ under the homomorphism $\varphi_s^{[m], \, \pm}$ is $\nu^{\widetilde{s}}\varphi_0^{[m], \, \pm}(\D_x^{\ho,\pm})$ (respectively $\nu^{\widetilde{s}}\varphi_{\frac{1}{2}}^{[m], \, \pm}(\D_x^{\ho,\pm})$), where $\nu$ was defined in \eqref{(14)} and $\widetilde{s}=s$ (respectively $\widetilde{s}=s-\frac{1}{2}$). Therefore, we will only consider $s=0, \ \frac{1}{2}$ throughout the article.

\noindent \textit{(c)} Observe that Proposition \ref{prop D and L} is the corrected version of Proposition 5.3 in \cite{bl}.
\end{obse}

Now we want to extend the homomorphism $\varphi_s^{[m], \, \pm}$ to a homomorphism between the central extensions of the corresponding Lie algebras. Define
\begin{equation*}
\eta_i(x,s)=\frac{e^{(s-1/2) x}+(-1)^ie^{-(s-1/2) x}}{2}\frac{x^i}{i!} \ \ \ (i\in\ZZ_+, \ s\in\CC).
\end{equation*}
The functions $\eta_i(x,s)$ satisfy:
\begin{eqnarray}
\eta_i(x,-s)&=&(-1)^i\eta_i(x,s+1), \label{(42)} \\
\eta_0(x,s+1/2)&=&\cosh(s x). \notag
\end{eqnarray}

\begin{prop} \label{prop extend fis}
The homomorphism $\varphi_s^{[m],\pm}$ lifts to a Lie algebra homomorphism $\fiss$ of the corresponding central extensions as follows:
\begin{eqnarray}
\fiss|_{(\widehat{\D}_x^{\pm})_j}&=&\varphi_s^{[m],\pm}|_{(\D_x^{\pm})_j} \ \ \text{(for $j\neq0$)}, \notag \\
\fiss(\sinh(xD))&=&\frac{1}{2}\sum_{i=0}^m\sum_{j\in\ZZ}\frac{\eta_i(x,s-j+1)-\eta_i(x,s-j)}{\sinh(x/2)}u^iE_{j,j}  \notag \\
                & &-\frac{1}{2}\sum_{i=0}^m\frac{\eta_i(x,s)}{\sinh(x/2)}u^ic_0+\frac{1}{2}\frac{\cosh(x/2)}{\sinh(x/2)}c_0, \label{(extend fis)} \\
\fiss(c_0)&=&1. \notag
\end{eqnarray}
\end{prop}

\begin{proof}
Note that $(\widehat{\D}_x^+)_0=(\widehat{\D}_x^-)_0$ and therefore $\widehat{\varphi}_s^{[m],+}|_{(\widehat{\D}_x^+)_0}=\widehat{\varphi}_s^{[m],-}|_{(\widehat{\D}_x^-)_0}$. See Proposition 5.2 in \cite{bl}.
\end{proof}

Let $\overline{m}=(m_1,\ldots,m_N)\in\ZZ_+^N$ and $\overline{s}=(s_1,\ldots,s_N)\in\CC^N$ be such that $s_i\in\ZZ$ implies $s_i=0$, $s_i\in\ZZ+\frac{1}{2}$ implies $s_i=\frac{1}{2}$ and $s_i \neq \pm s_j \ mod\ \ZZ$ for $i \neq j$, combining the Propositions \ref{prop D and gl}, \ref{prop D and c, d}, \ref{prop D and L} and \ref{prop extend fis} we obtain the following result.

\begin{prop} \label{prop extend iso sum D en sum g}
Given $\overline{m}$ and $\overline{s}$ as above, we have the following exact sequence of Lie algebras:
\begin{equation*}
0\rightarrow \bigcap_{i=1}^N J_{s_i}^{[m_i],\pm} \rightarrow \widehat{\D}_x^{\ho,\pm} \xrightarrow{\widehat{\varphi}_{\overline{s}}^{[\overline{m}], \, \pm}} \g_{\pm}^{[\overline{m}]} \rightarrow 0,
\end{equation*}
where \ $\widehat{\varphi}_{\overline{s}}^{[\overline{m}], \, \pm}=\bigoplus_{i=0}^N \widehat{\varphi}_{s_i}^{[m_i], \, \pm}$ and $\g_{\pm}^{[\overline{m}]}=\bigoplus_{i=0}^N \g_{\pm}^{[m_i]}$ with 
\[ \g_+^{[m_i]}=
\begin{cases}
  \widehat{\mathfrak{gl}}_{\infty}^{[m_i]}, & \text{if} \ s_i \neq 0, \ \frac{1}{2}, \\
  c_{\infty}^{[m_i]}, & \text{if} \ s_i=\frac{1}{2}, \\
  \li^{[m_i]}_+, & \text{if} \ s_i=0,
\end{cases}
\]
and
\[ \g_-^{[m_i]}=
\begin{cases}
  \widehat{\mathfrak{gl}}_{\infty}^{[m_i]}, & \text{if} \ s_i \neq 0, \ \frac{1}{2}, \\
  d_{\infty}^{[m]}, & \text{if} \ s_i=\frac{1}{2}, \\
  \li^{[m_i]}_-, & \text{if} \ s_i=0.
\end{cases}
\]
\end{prop}

\section{Realization of Quasifinite Height Modules of $\widehat{\D}_x^{\ho,\pm}$}

Let $\g^{[m]}$ stand for $\gl$ or $c_{\infty}^{[m]}$ or $d_{\infty}^{[m]}$ or $\li^{[m]}_{\pm}$. The proof of the following proposition is standard.

\begin{prop} \label{prop 3.1}
The $\g^{[m]}$-module $L(\g^{[m]}, \lambda)$ is quasifinite if and only if all but finitely many of the $^* h_k^{(i)}$ are zero, where $*$ represent $a$ or $c$ or $d$ or $\pm$ depending $\g^{[m]}$ is $\gl$ or $c_{\infty}^{[m]}$ or $d_{\infty}^{[m]}$ or $\li^{[m]}_{\pm}$.
\end{prop}

Let $\overline{m}=(m_1,\ldots,m_N)\in\ZZ_+^N$ and $\overline{s}=(s_1,\ldots,s_N)\in\CC^N$ be such that $s_i\in\ZZ$ implies $s_i=0$, $s_i\in\ZZ+\frac{1}{2}$ implies $s_i=\frac{1}{2}$ and $s_i \neq \pm s_j \ mod\ \ZZ$ for $i \neq j$, take a quasifinite $\lambda_i\in(\g^{[m_i]}_\pm)_0^*$ for each $i=1,\ldots,N$ and let $L(\g^{[m_i]}_\pm, \lambda_i)$ be the corresponding irreducible $\g^{[m_i]}_\pm$-module. Let $\overline{\lambda}=(\lambda_1,\ldots,\lambda_N)$. Then the tensor product
\begin{equation}
L(\g^{[\overline{m}]}_\pm,\overline{\lambda})=\bigotimes_{i=1}^N L(\g^{[m_i]}_\pm,\lambda_i),
\end{equation}
is a irreducible $\g^{[\overline{m}]}_\pm$-module, with $\g^{[\overline{m}]}_\pm=\bigoplus_{i=1}^N \g^{[m_i]}_\pm$ as in Proposition \ref{prop extend iso sum D en sum g}. The module $L(\g^{[\overline{m}]},\overline{\lambda})$ can be regarded as a $\widehat{\D}_x^{\pm}$-module via the homomorphism $\widehat\varphi_{\overline{s}}^{[\overline{m}],\pm}$, and will be denoted by $L_{\overline{s}}^{[\overline{m}]}(\overline{\lambda})$. We shall need the following proposition. Its proof is analogous to that of Proposition 4.3 \cite{kr1}.

\begin{prop} \label{prop 3.2}
Let $V$ a quasifinite $\widehat{\D}_x^{\pm}$-module. Then the action of $\widehat{\D}_x^{\pm}$ on $V$ naturally extends to the action of $(\widehat{\D}_x^{\ho,\pm})_k$ on $V$ for any $k \neq 0$.
\end{prop}

\begin{teor} \label{teo 3.3}
Let $V$ a quasifinite $\g^{[\overline{m}]}_\pm$-module, which is regarded as a $\widehat{\D}_x^{\pm}$-module via the homomorphism $\widehat
{\varphi}_{\overline{s}}^{[\overline{m}],\pm}$. Then any $\widehat{\D}_x^{\pm}$-submodule of $V$ is also a $\g^{[\overline{m}]}_\pm$-submodule. In particular, the $\widehat{\D}_x^{\pm}$-modules $L_{\overline{s}}^{[\overline{m}]}(\overline{\lambda})$ are irreducible if $\overline{s}=(s_1,\ldots,s_N)\in\CC^N$ is such that $s_i \in \ZZ$ implies $s_i=0, \ s_i \in \ZZ+\frac{1}{2}$ implies $s_i=\frac{1}{2}$, and $s_i\neq\pm s_j$ mod $\ZZ$ for $i \neq j$.
\end{teor}

\begin{proof}
Let $U$ be a $\widehat{\D}_x^{\pm}$-submodule of $V$, then U is a quasifinite $\widehat{\D}_x^{\pm}$-module as well, hence by Proposition \ref{prop 3.2} it can be extended to $(\widehat{\D}_x^{\ho,\pm})_k$ for any $k \neq 0$. By Proposition \ref{prop extend iso sum D en sum g}, the map $\widehat{\varphi}_{\overline{s}}^{[\overline{m}], \, \pm}:(\widehat{\D}_x^{\ho,\pm})_k \rightarrow (\mathfrak{g}_{\pm}^{[\overline{m}]})_k$ is surjective for any $k \neq 0$. Therefore $U$ is invariant with respect to all graded subspaces  $(\g^{[\overline{m}]}_\pm)_k$ ($k\neq 0$) of $\g^{[\overline{m}]}_\pm$. Using that $\g^{[\overline{m}]}_\pm$ coincides with its derived algebra, we finish the proof.
\end{proof}

\medskip

Given an irreducible highest weight $\widehat{\D}_x^{\pm}$-module $L(\widehat{\D}_x^{\pm},\lambda)$, using Theorem \ref{caractrizacion de MWH}, we have that it is quasifinite if and only if 
\begin{equation}\label{aaaaa}
    \Delta_\lambda(x)=\frac{d}{dx}
    \left(\frac{\phi_{\lambda}(x)}{2\sinh\left(\frac{x}{2}\right)}\right)
\end{equation}
where $\phi_{\lambda}(x)$ is an even quasipolynomial such that $\phi_{\lambda}(0)=0$.

 On the other hand, observe that a functional $\lambda\in(\widehat{\D}_x^{\pm})_0^*$ is also characterized by  $\Gamma_l=-\lambda(D^{l+1})$, where $l\in\ZZ_+^0$, and the central charge $\lambda(C)=c_0$, cf. \eqref{(31)}. Consider the new generating series:
\begin{equation}
\Gamma_{\lambda}(x)=\sum_{l\in\ZZ_+^0}\frac{x^{l+1}}{(l+1)!}\Gamma_l=-\lambda(\sinh(xD)), \label{(43)}
\end{equation}
observe que $\Gamma_{\lambda}(x)$ satisfies \eqref{(12)}, then using \eqref{(36)} we obtain
\begin{equation*}
\Gamma_{\lambda}(x)=\frac{\phi_{\lambda}(x)}{2\sinh\left(\frac{x}{2}\right)}.
\end{equation*}

We will show that in fact all the quasifinite $\widehat{\D}_x^{\pm}$-module can be realized as some $L_{\overline{s}}^{[\overline{m}]}(\g^{[\overline{m}]}_\pm,\lambda)$, and this is done by the study of  exponents and multiplicities  using the computation of the generating series $\Gamma_{m,s,\lambda}(x)$ of the highest weight  $\widehat{\D}_x^{\ho,\pm}$-module $L_s^{[m]}(\g^{[m]}_\pm,\lambda)$.

\begin{prop} \label{prop 3.4}
For $s \in (\CC-\ZZ/2)$, consider the embedding $\widehat{\varphi}_{{s}}^{[{m}]}:\widehat{\D}_x^{\pm} \rightarrow \gl$. Then the $\gl$-module $L(\gl,\lambda)$ regarded as a $\widehat{\D}_x^{ \pm}$-module via $\widehat{\varphi}_{{s}}^{[{m}]}$ is isomorphic to $L(\widehat{\D}_x^{ \pm};e^+,e^-)$ where $e^+, \ e^-$ consist of exponents $(s-j-\frac{1}{2})$ with $j \in \ZZ$ and multiplicities
\begin{equation*}
\sum_{\substack{0\leqslant i \leqslant m, \\ i \ \text{even}}} \frac{^ah_j^{(i)}x^i}{i!} \ \ \ \text{and} \ \ \
\sum_{\substack{0 \leqslant i \leqslant m, \\ i \ \text{odd}}} \frac{^ah_j^{(i)}x^i}{i!},
\end{equation*}
respectively.
\end{prop}

\begin{proof} By Proposition \ref{prop 3.1} and Theorem \ref{teo 3.3}, the $\widehat{\D}_x^{ \pm}$-module $L(\gl,\lambda)$ is an irreducible quasifinite highest weight module. Using (\ref{(extend fis)}), the central charge $c=c_0$.  
 Using the explicit expression of the homomorphism $\widehat{\varphi}_s^{[m]}:\widehat{\D}_x^{\pm} \rightarrow \gl$ given in Proposition \ref{prop extend fis}), and the formulas \eqref{(43)}, \eqref{(extend fis)} and \eqref{(40)} we have that
\begin{eqnarray*}
\Gamma_{m,s,\lambda}(x) & = & -\lambda(\widehat{\varphi}_s^{[m]}(\sinh(xD))) \\
& = & \frac{1}{2}\sum_{i=0}^m \sum_{j\in\ZZ} \frac{\eta_i(x,s-j)}{\sinh(x/2)} \, ^ah_j^{(i)} -\frac{1}{2}\frac{\cosh(x/2)}{\sinh(x/2)}c_0.
\end{eqnarray*}
Now  the proposition follows from the definition of exponents and their multiplicities.
\end{proof}

\begin{prop} \label{prop 3.5} 
Consider $\g_+^{[m]}=c_{\infty}^{[m]}$,  $\g_-^{[m]}=d_{\infty}^{[m]}$ and  the embedding $\widehat{\varphi}_{\frac{1}{2}}^{[{m}], \, \pm}:\widehat{\D}_x^{\pm} \rightarrow \g_{\pm}^{[m]}$. 
Then the $\g_{\pm}^{[m]}$-module $L(\g_{\pm}^{[m]},\lambda)$ regarded as a $\widehat{\D}_x^{ \pm}$-module via $\widehat{\varphi}_{\frac{1}{2}}^{[{m}], \, \pm} $ is isomorphic to $L(\widehat{\D}_x^{ \pm}; e^+,e^-)$ where $e^+, \ e^-$ consist of exponents $j \in \ZZ_+$ and multiplicities
\begin{equation*}
\sum_{\substack{0\leqslant i \leqslant m, \\ i \ \text{even}}} \frac{^gh_j^{(i)}x^i}{i!} \ \ \ \text{and} \ \ \
\sum_{\substack{0 \leqslant i \leqslant m, \\ i \ \text{odd}}} \frac{^gh_j^{(i)}x^i}{i!},
\end{equation*}
respectively, where $^gh_j^{(i)}=0$ for $i$ odd and $g$ represents $c$ or $d$ depending of $\g_{\pm}^{[m]}$.
\end{prop}

\begin{proof} We will only need to compute $\Gamma_{m,s,\lambda}(x)$. The rest of the statement is clear, cf. the proof of Proposition \ref{prop 3.4}. 
Recall Remark \ref{obse1} \textit{(a)} and consider  the explicit computation of the homomorphism $\widehat{\varphi}_{\frac{1}{2}}^{[m], \, \pm}:\widehat{\D}_x^{\pm} \rightarrow \g_{\pm}^{[m]}$  given in Proposition \ref{prop extend fis}. Using \eqref{(43)}, \eqref{(42)}, \eqref{(extend fis)} and \eqref{(41)} we have that
\begin{eqnarray*}
\Gamma_{m,s,\lambda}(x)&=&-\lambda(\widehat{\varphi}_{\frac{1}{2}}^{[m], \, \pm}(\sinh(xD))) \\
&=& \frac{1}{2}\sum_{i=0}^m \sum_{j>0} \frac{\eta_i(x,j+1/2)}{\sinh(x/2)} \, ^gh_j^{(i)} \\
& &+\frac{1}{2}\sum_{\substack{0 \leqslant i \leqslant m \\ i \ \text{even}}} \frac{\eta_i(x,1/2)}{\sinh(x/2)} \, ^gh_0^{(i)}-\frac{1}{2}\frac{\cosh(x/2)}{\sinh(x/2)}c_0,
\end{eqnarray*}
which proves the proposition.
\end{proof}

\begin{prop} \label{prop 3.6}
Consider the embedding $\widehat{\varphi}_{0}^{[{m}], \, \pm}:\widehat{\D}_x^{\pm} \rightarrow \li^{[m]}_{\pm}$. 
Then  
the $\li^{[m]}_{\pm}$-module $L(\li^{[m]}_{\pm},\lambda)$ regarded as a $\widehat{\D}_x^{ \pm}$-module via $\widehat{\varphi}_{0}^{[{m}], \, \pm}$ is isomorphic to $L(\widehat{\D}_x^{ \pm}; e^+,e^-)$ where $e^+, \ e^-$ consist of exponents $-j-\frac{1}{2}$ with $j\in\ZZ_+$ and multiplicities
\begin{equation*}
\sum_{\substack{0\leqslant i \leqslant m, \\ i \ \text{even}}} \frac{^{\pm}h_j^{(i)}x^i}{i!} \ \ \ \text{and} \ \ \
\sum_{\substack{0 \leqslant i \leqslant m, \\ i \ \text{odd}}} \frac{^{\pm}h_j^{(i)}x^i}{i!}.
\end{equation*}
respectively.
\end{prop}

\begin{proof}
Recall Remark \ref{obse1} \textit{(a)} and consider the explicit computation of the homomorphism $\widehat{\varphi}_0^{[m], \, \pm}:\widehat{\D}_x^{\pm} \rightarrow \li^{[m]}_{\pm}$ obtained in Proposition \ref{prop extend fis}. Using \eqref{(43)}, \eqref{(42)}, \eqref{(extend fis)} and \eqref{(44)} we have that
\begin{eqnarray*}
\Gamma_{m,s,\lambda}(x)&=&-\lambda(\widehat{\varphi}_0^{[m], \, \pm}(\sinh(xD))) \\
&=& \frac{1}{2}\sum_{i=0}^m \sum_{j\geqslant0} \frac{\eta_i(x,-j)}{\sinh(x/2)} \, ^{\pm}h_j^{(i)}-\frac{1}{2}\frac{\cosh(x/2)}{\sinh(x/2)}c_0,
\end{eqnarray*}
which proves the proposition.
\end{proof}

Take an irreducible quasifinite highest weight $\widehat{\D}_x^{\pm}$-module $L(\widehat{\D}_x^{\pm},\lambda)$ with central charge $c_0$ and
\begin{equation*}
\Gamma_\lambda(x)=\frac{\phi_{\lambda}(x)}{2\sinh(x/2)},
\end{equation*}
where $\phi_{\lambda}(x)$ an even quasipolynomial with $\phi_{\lambda}(0)=0$. We will write
\begin{equation}
\phi_{\lambda}(x)+\cosh(x/2)c_0=\sum_{s\in\CC}\sum_{i=1}^{m_s} a_{s,i} \, \eta_i(x,s), \label{(34)}
\end{equation}
where $a_{s,i}\in\CC$ and $a_{s,i} \neq 0$ for only finitely many $s\in\CC$. Since, by definition of $\eta_i$, we have that $\eta_i(x,-s)=(-1)^i\eta_i(x,s+1)$, to avoid ambiguities in the expression of $\phi_{\lambda}(x)$ above, we will choose the parameter $s$ following these rules: when $s\in\ZZ$ we require $s\leqslant0$; when $s\in\ZZ+\frac{1}{2}$, we ask $s\leqslant\frac{1}{2}$, when $s\notin\ZZ/2$, we require that $\mathrm{Im}\, s>0$ if $\mathrm{Im}\, s\neq0$ or $s-[s]<\frac{1}{2}$ if $s\in\RR$, where $\mathrm{Im}\, s$ is the imaginary part of $s$, and $[s]$ denotes the biggest integer smaller than $s$ respectively.

Decompose the set $\{s\in\CC|a_{s,i}\neq0 \ \text{for some} \ i\}$ into a disjoint union of equivalence classes under the equivalence relation $s\thicksim s'$ if and only if $s=\pm s' \ (mod\,\ZZ)$. Pick a representative $s$ in an equivalence class $S$ such that $s=0$ if the equivalence class is in $\ZZ$ and $s=\frac{1}{2}$ if the equivalence class is in $\ZZ+\frac{1}{2}$. Let $S=\{ s, \, s-k_1, \, s-k_2,\ldots \}$ be such an equivalence class and take $m=max_{s\in S}\, m_s$. Put $k_0=0$. It is easy to see that if $s=0$ or $\frac{1}{2}$, then $k_i\in\NN$.

We associate to $S$ the $\g^{[m]}_\pm$-module $L_S^{[m]}(\g^{[m]}_\pm,\lambda_S)$ in the following way: if $s\notin\ZZ/2$, let $^ah_{k_r}^{(i)}=a_{s+k_r,i}$ with $i=0,\ldots,m_s$ and $r=0,1,2,\ldots$. We associate to $S$ the $\gl$-modules $L_S^{[m]}(\gl,\lambda_S)$ with central charges and labels
\begin{equation*}
c_i=\sum_{k_r} {^ah_{k_r}^{(i)}}, \ \ \ ^a\lambda_j^{(i)}=\sum_{k_r\geqslant j} (^ah_{k_r}^{(i)}-\delta_{k_r,0}\, \ c_i).
\end{equation*}
If $s=\frac{1}{2}$, let $^gh_{k_r}^{(i)}=a_{\frac{1}{2}+k_r,i}$, with $i=0,\ldots,m_\frac{1}{2}$ and $r=0,\, 1, \,2,\ldots$. We associate to $S$ the $\g_{\pm}^{[m]}$-module $L_S^{[m]}(\g_{\pm}^{[m]},\lambda_S)$ with central charges and labels
\begin{equation*}
c_i=\sum_{k_r} {^gh_{k_r}^{(i)}} \  \text{($i$ even),} \ c_i=0 \ \ \text{($i$ odd),} \ \ ^g\lambda_j^{(i)}=\sum_{k_r\geqslant j} {^gh_{k_r}^{(i)}},
\end{equation*}
where $\g_+^{[m]}=c_{\infty}^{[m]}, \ \g_-^{[m]}=d_{\infty}^{[m]}$ and $g$ represents $c$ or $d$ depending of $\g_{\pm}^{[m]}$, $j\in\NN$, $i=0,\ldots,m_\frac{1}{2}$. Similarly if $s=0$, $^{\pm}h_{k_r}^{(i)}=a_{k_r,i}$, with $i=0,\ldots,m_0$ and $r=0,\, 1, \,2,\ldots$. We associate to $S$ the $\li_{\pm}^{[m]}$-module $L_S^{[m]}(\li_{\pm}^{[m]},\lambda_S)$ with central charges and labels 
\begin{equation*}
c_i=\sum_{k_r} {^{\pm}h_{k_r}^{(i)}}, \ \ \ ^{\pm}\lambda_j^{(i)}=\sum_{k_r\geqslant j} ({^{\pm}h_{k_r}^{(i)}}-\delta_{k_r,0}\, \ c_i).
\end{equation*}
Denote by $\{ s_1, \, s_2,\ldots,s_N \}$ a set of representative of equivalence classes in the set $\{ s\in\CC|a_{s,i}\neq0 \ \text{for some} \ i \}$. By Theorem \ref{teo 3.3}, the $\widehat{\D}_x^{\pm}$-module $L_{\overline{s}}^{[\overline{m}]}(\g^{[m]},\overline{\lambda})$ is irreducible for $\overline{s}=(s_1,\ldots,s_N)$ such that $s_i\in\ZZ$ implies $s_i=0$, $s_i\in\ZZ+\frac{1}{2}$ implies $s_i=\frac{1}{2}$ and $s_i\neq\pm s_j \ (mod\, \ZZ)$ for $i\neq j$. Then we have

\begin{equation*}
\Gamma_{\overline{m},\overline{s},\overline{\lambda}}(x)=\sum_i \Gamma_{m_i,s_i,\lambda_i}(x), \ \ c=\sum_i c_0^{(i)}.
\end{equation*}

Using Theorem \ref{teo 3.3} and Proposition \ref{prop 3.4}, \ref{prop 3.5} and \ref{prop 3.6}, we have proved the following result.

\begin{teor} \label{teor ultimo}
Let $V$ be an irreducible quasifinite highest weight $\widehat{\D}_x^{\pm}$-module with highest weight $\lambda$, central charge $c_0$ and
\begin{equation*}
\Gamma_\lambda(x)=\frac{\phi_{\lambda}(x)}{2\sinh(x/2)}
\end{equation*}
with $\phi_{\lambda}(x)$ an even quasipolynomial such that $\phi_{\lambda}(0)=0$, which is written in the form \eqref{(34)}. Then $V$ is isomorphic to the tensor product of the modules $L_S^{[m]}(\g^{[m]},\lambda_S)$ with distinct equivalence classes $S$.
\end{teor}

\begin{obse}
A different choice of the representative $s\notin\ZZ/2$ has the effect of shifting $\gl$ via the automorphism $\nu^i$ for some $i$. It is easy to see that any irreducible quasifinite highest weight module $L(\widehat{\D}_x^{\ho,\pm},\lambda)$ can by obtained as above as a unique way up to the shift $\nu$.
\end{obse}

\end{document}